\documentclass[10pt]{amsart}
\usepackage{amssymb, mathrsfs, bbm}
\usepackage{appendix}
\usepackage[T1]{fontenc}
\usepackage[sc, osf]{mathpazo}
\selectfont
\usepackage[protrusion=true, expansion=true]{microtype}

\usepackage[all, arc]{xy}
\SelectTips{eu}{}
\entrymodifiers={+!!<0pt,\fontdimen22\textfont2>}

\usepackage[pdftex, colorlinks=true, linkcolor=blue, citecolor=magenta]{hyperref}

\theoremstyle{plain}
\newtheorem{thm}[subsubsection]{Theorem}

\newtheorem{prop}[subsubsection]{Proposition}
\newtheorem{cor}[subsubsection]{Corollary}

\theoremstyle{definition}

\theoremstyle{remark}

\theoremstyle{definition}

\numberwithin{equation}{subsubsection}

\def\cM{\mathcal{M}}

\def\11{\mathbf{1}}

\def\AA{\mathbf{A}}

\def\ZZ{\mathbf{Z}}

\def\id{\mathrm{id}}

\def\Tr{\mathrm{Tr}}

\newcommand{\const}[1]{\mathbb{1}_{#1}}

\title{On Euler-Poincar\'e characteristics}
\author{R. Virk}
\email{rsvirk@gmail.com}
\address{The Appalachians}
\begin{document}
\maketitle
\renewcommand{\thesubsection}{\arabic{subsection}}
\subsection{Introduction}
It is a remarkable fact about algebraic varieties that the Euler characteristic of the cohomology of a complex algebraic variety coincides with the Euler characteristic of the cohomology with compact support. Using Whitney stratifications, one may extend this to cohomology with coefficients in any algebraically constructible sheaf.
These statements cry out for a generalization to the relative setting. This, and a bit more, is what this note aims to do.

The main result is Theorem \ref{main}; some applications are discussed in \S\ref{s4}-\S\ref{s6}. Theorem \ref{main} is the topological analogue of an old result of G. Laumon in $\ell$-adic cohomology \cite{L}. The argument presented here uses the same local to global technique as \cite{L}. Thus, all the main ideas are due to Laumon. After this note was written J. Sch\"urmann informed me that the non-mixed case of Theorem \ref{main} is also covered in \cite[\S6.0.6]{Sc} (with a similar proof). 

In Theorem \ref{maineq} we extend Theorem \ref{main} to quotient stacks (i.e., to the setting of equivariant cohomology). In the setting of $\ell$-adic cohomology, and for the special case of quotients by finite groups, this is due to I. Illusie and W. Zheng \cite[Theorem 1.3]{IZ}.

\subsection{The main result}
For a complex variety $X$, let $DX$ be either:
\begin{enumerate}
\item $D^b(\mathrm{MHM}(X))$ - M. Saito's bounded derived category of mixed Hodge modules on $X$;
\item $D^b_c(X_{\mathrm{an}})$ - the bounded derived category of algebraically constructible sheaves of $A$-modules on the complex analytic site associated to $X$. Here $A$ is a commutative ring (with $1$) of finite global dimension.
\end{enumerate}
(i) will be referred to as the mixed case, (ii) as the non-mixed case.

For $\cM\in DX$, let $[\cM]$ denote the class of $\cM$ in the Grothendieck ring of $DX$. Let $(1)$ denote Tate twist in the mixed case, or the identity functor in the non-mixed case.
Write $\widetilde{K}X$ for the quotient of the Grothendieck ring of $DX$ by the ideal generated by elements of the form $[\cM(1)]-[\cM]$.

Let $f\colon X\to Y$ be a morphism of varieties. Associated to $f$ are the derived functors $f_*,f_!\colon DX \to DY$. These induce group morphisms $f_*,f_!\colon \widetilde KX \to \widetilde KY$.
\begin{thm}\label{main}The maps
$f_*, f_!\colon \widetilde KX \to \widetilde KY$
coincide.
\end{thm}
\begin{proof}
Using Nagata compactification we may factor $f$ as $f = p\circ j$ with $p$ proper, and $j$ an open immersion. Thus, it suffices to only consider open immersions. 

Let $j\colon U\hookrightarrow X$ be an open immersion, and $i\colon Z \hookrightarrow X$ the inclusion of the closed complement $Z:= X - U$. The distinguished triangle
$j_!j^* \to \id \to i_*i^*$ applied
to $j_*$ yields the triangle
$j_! \to j_* \to i_*i^*j_*$. 
So it suffices to show $i^*j_*=0$ in $\widetilde KZ$.

Let $\mathrm{Bl}_ZX$ be the blow up of $X$ along $Z$. Then we have a commutative diagram, with Cartesian squares and proper vertical arrows,
\[\xymatrix{ \mathrm{Bl}_Z X\times_X Z \ar[d]_-{\tilde\pi}\ar[r]^-{\tilde i} & \mathrm{Bl}_Z X \ar[d]_-{\pi}&\ar[l]_-{\tilde j}U\ar@{=}[d] \\ 
Z \ar[r]^-{i} & X & \ar[l]_-{j}U
\ }\]
By proper base change, $i^*j_* = i^*\pi_*\tilde j_* = \tilde\pi_* \tilde i^*\tilde j_*$.
Hence, we may assume $Z$ is an effective Cartier divisor.

Let $\{U_1, \ldots, U_n\}$ be a Zariski open cover of $X$ such that, for each $i$, $Z\cap U_i$ is a principal divisor in $U_i$. For each subset $I\subseteq \{1, \ldots, n\}$,
set $V_I := \bigcap_{i\in I} V_i$, $U_I := U\cap V_I$, and $Z_I:= Z\cap V_I$. Let $u_I\colon U_I\hookrightarrow U$, $j_I\colon U_I\hookrightarrow X$, $z_I\colon Z_I\hookrightarrow Z$, and $i_I\colon Z_I\hookrightarrow X$ be the evident inclusions.
Then the usual argument, \`a la \v Cech, yields
\[ i^*j_* = \sum_{\substack{I\subseteq \{1, \ldots, n\}, \\ I\neq \varnothing}} (-1)^{|I|-1} z_{I*}i_I^*j_{I*}u_I^* \]
in the Grothendieck group. Hence, we may further assume that $Z$ is principal.

Let $f\colon X \to \AA^1$ be a defining equation for $Z$. 
Then we have a distinguished triangle
$\psi_f \to \psi_f(-1) \to i^*j_*$,
where 
$\psi_f\colon DU \to DX$
is the (unipotent part of the) nearby cycles functor\footnote{
The functor $\psi_f$ is shifted from the usual nearby cycles by $[-1]$ so as to preserve the perverse t-structure. This is purely personal preference, and irrelevant to the discussion at hand.}
 associated to $f$, and
$(-1)$ is the inverse of $(1)$.
\end{proof}

Clearly, the proof of Theorem \ref{main} holds in any `sheaf setting' admitting a partial yoga of the six functors and an appropriate formalism of nearby cycles. Possible candidates (other than those under consideration): holonomic (not necessarily regular) D-modules, twistor D-modules, p-adic cohomology, motivic sheaves, etc. 
%
%
\subsection{Exceptional pullback}\label{s4}
To a morphism of varieties $f\colon X\to Y$ we may also associate the triangulated functors $f^*,f^!\colon DY \to DX$. Denote the induced group morphisms $\widetilde KY \to \widetilde KX$ by $f^*$ and $f^!$ also.
\begin{prop}The maps $f^*,f^!\colon \widetilde KY \to \widetilde KX$ coincide.
\end{prop}
\begin{proof}
First, we handle the case of a closed immersion.
Let $i\colon D \hookrightarrow Z$ be a closed immersion and $j\colon U \hookrightarrow Z$ the complementary open immersion. Apply $i^*$ to the distinguished triangle $i_*i^!\to \id \to j_*j^*$ to get the triangle $i^! \to i^* \to i^*j_*j^*$. By Theorem \ref{main}, $i^*j_* = i^*j_! = 0$. So $i^! = i^*$.

Now suppose $f\colon X\to Y$ is a morphism with $X$ smooth. Via its graph, $f$ factors as a closed immersion $X \hookrightarrow X\times Y$ followed by the projection $p\colon X\times Y\to Y$. As $X$ is smooth, $p^! = p^*$ in $\widetilde K(X\times Y)$. Consequently, $f^!=f^*$.

We now deal with the general case. Stratify $X$ by closed subvarieties 
\[ \varnothing= X_{-1} \subset X_0 \subset \cdots \subset X_n = X \]
such that each $Z_i := X_i - X_{i-1}$ is smooth. Let $j_i \colon Z_i \hookrightarrow X$ be the inclusion, and set $f_i :=f \circ j_i$. Then
\[
f^! = \sum_{i=1}^n j_{i*}f_i^! = \sum_{i=1}^n j_{i*}f_i^* =\sum_{i=1}^n j_{i!}f_i^* 
= f^* \qedhere \]
\end{proof}
\begin{cor}Let $\const{X} \in DX$ be the monoidal unit (i.e., the `constant sheaf') and let $\omega_X\in DX$ be the dualizing object. Then $[\const{X}] = [\omega_X]$ in $\widetilde KX$.
\end{cor}

\subsection{Quotient stacks}
Suppose a linear algebraic group $G$ acts on $X$. Let $D_GX$ denote the analogue of $DX$ for the stack $[X/G]$ (for instance, see \cite{BL}).
Associated to a $G$-equivariant morphism $f\colon X\to Y$, one has functors $f^*, f_*, f_!, f^!$ between $D_GX$ and $D_GY$ satisfying the usual formalism.
Write $\widetilde K_GX$ for the Grothendieck ring of $D_GX$ modulo the same ideal as before.
\begin{thm}\label{maineq}The induced maps $f_*,f_!\colon \widetilde K_GX\to \widetilde K_GY$ coincide. Similarly, the induced maps $f^!, f^*\colon \widetilde K_GY \to \widetilde K_GX$ coincide.
\end{thm}

\begin{proof}We will only demonstrate the assertion for $f_*,f_!$. The proof for $f^*,f^!$ is similar and left to the reader.

Using Steifel varieties, as in \cite[\S3.1]{BL}, one sees that property ($\ast$) of \cite[\S2.2.4]{BL} holds for $Y$ (alternatively, see \cite[Lemme 18.7.5]{LMB}). In particular, we obtain a $G$-torsor $Z \to \overline{Z}$ along with a smooth equivariant morphism $p\colon Z \to Y$ of some relative dimension $d$ with connected fibres.\footnote{
There is a small subtlety here. In general, $\overline{Z}$ (which is constructed as an associated bundle) is only an algebraic space and not necessarily a variety. In the non-mixed setting this is a non-issue. In the mixed setting there are two ways to deal with this. One can either restrict to only quasi-projective varieties, in which case $\overline{Z}$ is a variety. Or, one may observe that as mixed Hodge modules are \'etale local, the formalism of $DX$ and $D_GX$ extends to algebraic spaces. However, note that the details of such an extension are not yet available in the literature.
}
Let $\cM_GY$, $\cM_GZ$ denote the hearts of the perverse t-structure on $D_GY$ and $D_GZ$ respectively.
Then the functor $p^*[d]\colon \cM_GY \to \cM_GZ$ is full, faithful, and its essential image is closed under taking subquotients.
Consequently, $p^*\colon \widetilde K_GY \to \widetilde K_GZ$ is injective.
Thus, using smooth base change, we may assume that there exist $G$-torsors $X\to \overline{X}$ and $Y\to \overline{Y}$. But then we have canonical equivalences $D_GX \simeq D\overline{X}$ and $D_GY\simeq D\overline{Y}$ which commute with the functors $f_*$ and $f_!$. Hence, the assertion reduces to the non-equivariant statement of Theorem \ref{main}.
\end{proof}

The above proof applies more generally to `Bernstein-Lunts stacks' (see \cite[\S18.7.4]{LMB}). However, I do not know of any examples of such stacks other than quotient stacks.

\subsection{Traces}\label{s6}In this section we consider the non-mixed case exclusively. Further, take sheaves to be sheaves of vector spaces over some fixed field.

If $G$ is a finite group, then $D_GX$ does coincide with the na\"ive derived category of equivariant sheaves \cite[\S8]{BL}. In particular, the derived category of the classifying stack $[\ast/G]$ is the derived category of finite dimensional $k[G]$-representations. 

For each $g\in G$, $\cM\in D_GX$, and $k\in \ZZ$, consider the traces
\[ \Tr(g, H^k(X; \cM)) \quad \mbox{and} \quad \Tr(g, H^k_c(X; \cM)) \]
of the $g$-action on cohomology and compactly supported cohomology with coefficients in $\cM$. Set
\begin{align*}
\chi_g(\cM) &:= \sum_k (-1)^k\Tr(g, H^k(X; \cM)), \\
\chi_{g,c}(\cM) &:=\sum_k (-1)^k \Tr(g, H^k_c(X; \cM)).
\end{align*}
\begin{cor}If $G$ is a finite group, then $\chi_g(\cM) = \chi_{g,c}(\cM)$.
\end{cor}
%
%

\end{document}